\documentclass{amsart}
\usepackage{amsmath}%用于Align环境
\usepackage{amsfonts}%用于\mathbb{}
\usepackage{amsthm}%用于定理环境
\usepackage{amssymb}%第一次用于写右箭头\nrightarrow 命令
\usepackage{enumerate}%计时器，用于编号公式之类;还用于列表环境。
\usepackage[pagewise]{lineno}%\linenumbers
\usepackage{cite}
\numberwithin{equation}{section}
\theoremstyle{plain}
\newtheorem{thm}{Theorem}[section]%写定理

\newtheorem{proposition}[thm]{Proposition}

\newtheorem{cor}[thm]{Corollary}
\newtheorem{lemma}[thm]{Lemma}
\theoremstyle{definition}%定义格式
\newtheorem{defn}[thm]{Definition}
\theoremstyle{remark}%注释格式
\newtheorem{rmk}[thm]{Remark}

\makeatletter%以下四行是用于写大小写罗马数字的，本文暂时只用得到大写。

\newcommand{\Rmnum}[1]{\expandafter\@slowromancap\romannumeral #1@}
\makeatother

\thanks{The author is supported by the Young Scientist Program of the Ministry of Science and Technology of China No. 2021YFA1002200.}
\begin{document}

\title{Estimates of the Kobayashi metric and Gromov hyperbolicity on convex domains of finite type}
\author{ Hongyu Wang}

\address{ School of Science, Beijing University of Posts and Telecommunications, Beijing 100876, China}

\email{ wanghyu@bupt.edu.cn}

\begin{abstract}
In this paper we give an local estimate for the Kobayashi distance on a bounded convex domain of finite type, which relates to a local pseudodistance near the boundary. The estimate is precise up to a bounded additive term. Also we conclude that the domain equipped with the Kobayashi distance is Gromov hyperbolic which gives another proof of the result of Zimmer.

\end{abstract}

\maketitle

\section{\noindent{{\bf Introduction}}}
The Kobayashi metric is an important invariant metric in several complex variables. Although in general there is no exact formula for this metric , much progress has been made in the study of the boundary behaviour. For the latest results for the estimates we refer to the survey\cite{2005Invariant}.  Much less is known about the boundary behaviour of the distance function of the Kobayashi metric. In \cite{balogh2000gromov}, Balogh and Bonk obtained a sharp estimate of the Kobayashi metric on bounded strongly pseudoconvex domains in $\mathbb{C}^n$. Also they related the distance function with the Carnot-Caratheodory metric on the boundary. Later Herbort \cite{2005Estimation} studied bounded pseudoconvex domains of finite type in $\mathbb{C}^2$; he related the distance function to a "pseudodistance" function introduced by Catlin\cite{1989Estimates} and obtained effective estimates of the Kobayashi metric and other invariant metrics. Following their work, in this paper we will study the Kobayashi metric on bounded convex domains of finite type in $\mathbb{C}^n$ and give a precise estimate of the Kobayashi distance locally up to a bounded additive term.

Let $\Omega$ be a smoothly bounded domain and convex in some neighbourhood $U$ of $p\in\partial\Omega$ and suppose $p$ is a point of finite type in sense of D'Angelo. After perhaps shrinking $U$, McNeal \cite{1994Estimates} defined a local pseudodistance $M(\cdot,\cdot)$ in $\Omega\cap U$(see Section 3.2 for the definition), and
we introduce the function $g:\Omega\cap U\times\Omega\cap U\rightarrow \mathbb{R}$ by
$$g(x,y)=\log\left[\frac{M(x,y)+\delta_{\Omega}(x)\vee\delta_{\Omega}(y)}{\sqrt{\delta_{\Omega}(x)\delta_{\Omega}(y)}}\right]$$
where $M$ is the local pseudodistance, $a\vee b:=\max\{a,b\}$ and $\delta_{\Omega}(x)$ is the shortest distance from $x$ to $\partial\Omega$.

Our main result is the following:

\begin{thm}\label{main}
Let $\Omega\subset\mathbb{C}^n$ be smoothly bounded and convex in some neighborhood of $p\in\partial\Omega$ and suppose that $p$ is a point of finite type in sense of D'Angelo. Then there exists a neighborhood $U$ of $p\in\partial\Omega$ and a constant $C$ such that for any $x,y\in\Omega\cap U$,
$$g(x,y)-C\leq K_{\Omega}(x,y)\leq g(x,y)+C.
$$
where $K_{\Omega}$ is the Kobayashi distance.

\end{thm}

Following Balogh and Bonk, first we guess the "almost geodesics" between two distinct points. Then for the upper bound, the key point of the proof is to take advantage of the properties of polydisks constructed by Chen\cite{chen} and McNeal\cite{1994Estimates} and the convexity of the domain.
For the lower bound, inspired by Herbot\cite{2005Estimation}, we adapt the steps in \cite{balogh2000gromov} to this case by using the pseudodistance property of $M$.

%Obviously $|r(x)|$ can be substituted by $\delta_{\Omega}(x)$, since $\nabla r\neq 0$ by the definition of defining functions. But we still use $|r|$ because the pseudodistance $M(\cdot,\cdot)$ depends on the defining function $r$.

In \cite{balogh2000gromov}, Balogh and Bonk investigated the {\it Gromov hyperbolicity} of the Kobayashi metric for bounded strictly pseudoconvex domains in $\mathbb{C}^n$, and they also gave a correspondence between the Gromov boundary and the Euclidean boundary which can be used in boundary extension problems for proper holomorphic mappings and quasi-isometric mappings. Later Zimmer \cite{Zimmer2016Gromov} obtained the Gromov hyperbolicity on bounded convex domains of finite type. Recently Fiacchi\cite{2021Gromov} obtained the Gromov hyperbolicity on bounded pseudoconvex domains of finite type in $\mathbb{C}^2$. Following the discussion in \cite{balogh2000gromov}, we show the Gromov hyperbolicity on  bounded convex domains of finite type which gives another proof for the result obtained by Zimmer. Actually we can also get a correspondence between the Gromov boundary and the Euclidean boundary of the domain.

 \begin{thm}\label{main2}
 Let $\Omega$ be a bounded convex domain with smooth boundary of finite type, then $(\Omega,K_{\Omega})$ is a Gromov hyperbolic space.
 \end{thm}
By calculus, we can show the Gromov hyperbolicity holds 'locally' since the estimate of the Kobayashi metric in Theorem \ref{main} is local. Then we need to get the 'global' Gomov hyperbolicity from the 'local' Gromov hyperbolicty. Similar interesting problems have been recently studied by Bracci, Gaussier, Nikolov and Thomas \cite{2022Local}. Their idea is to consider the stability of the quasi-geodesics and the proof can be adapted to this case. For the reader's convenience, we supply a short proof by using the Thin-triangle property of the Gromov hyperbolicity.

  As an application of Theorem \ref{main2}, we show the following boundary extension result for proper holomorphic mappings:
 \begin{cor}\label{1.3}
 Let $\Omega_{i}(i=1,2)$ be bounded convex domains with smooth boundary of finite type, and let $f:\Omega_1\rightarrow\Omega_2$ be a proper holomorphic map.
 %If $p\in\partial\Omega_1$ and $q\in\partial \Omega_2$ such that $\exists \{x_{k}\}\in\Omega_1$, $x_k\rightarrow p$, and $f(x_k)\rightarrow q$.
 If there exists $p_i\in\partial\Omega_i(i=1,2)$ and a sequence $\{x_j\}\in\Omega_1$ such that $\lim\limits_{j\rightarrow \infty} x_j=p_1$,and $\lim\limits_{j\rightarrow\infty}f(x_j)=p_2$, then there exists open sets $V_{i}\ni p_i $ and the local pseudodistances $M_i(i=1,2)$ such that $f$ extends to a continuous map $\bar{f}$ in $\overline{V}_1\cap \overline{\Omega}_1$ and
for any $\xi,\eta\in \overline{V}_1\cap \partial\Omega_1,$
$$M_2(\bar{f}(\xi),\bar{f}(\eta))\lesssim M_{1}(\xi,\eta)$$
 \end{cor}
 Suppose each boundary point is of finite type at most $L$. It is well known that the extension $\bar{f}$ is H\"{o}lder continuous with exponent $1/L$,  which can be easily deduced by the Corollary and the estimate of the local pseudodistance(see Section 3.2).

 The paper is organized as follows. In Section 2 we recall some definitions and preliminary results. Section 3 focuses on the constructions of the polydisk and local pseudodistance defined by McNeal. In Section 4 we give the proof of the Theorem \ref{main}. In Section 5 we recall some background on Gromov hyperbolic spaces and give the proof of Theorem \ref{main2}. Finally we present some applications concerning boundary extensions of mappings.

%
% To prove the above theorem, we need the construction of the polydisk and estimates of the Bergman kernel in convex domains of finite type due to Chen\cite{chen} and McNeal\cite{1994Estimates} (needs a little correction, see \cite{Nikolov2013On}) which is mainly used in the proof related to (3).
%
%
%
%
%
%
%
%
%
%\bigskip
\section{\noindent{{\bf Preliminaries}}}
\subsection{Notation}\
(1) \:For $z\in\mathbb{C}^n$, let $|\cdot|$  denote the standard Euclidean
norm, and let $|z_1-z_2 |$ denote the standard Euclidean
distance of $z_1, \:z_2\in \mathbb{C}^n$.

(2) \:Given an open set $\Omega\subsetneq\mathbb{C}^n,\:x\in
\Omega$ and $v\in\mathbb{C}^n$, denote
$$\delta_{\Omega}(x)=\inf\left\{|x- \xi|:\xi\in\partial
\Omega\right\}.$$
Let $H:=x+\mathbb{C}\cdot v$, then denote
$$\delta_{\Omega}(x,v)=\inf\left\{|x- \xi|:\xi\in\partial
\Omega\cap H\right\}.$$

(3) \:For any curve $\gamma: I \rightarrow \mathbb C^n$, we denote the Kobayashi length by $L_k(\gamma)$.

(4) \:For all real numbers $a,b$, we denote $a\vee b:= \max\{a, \: b\}$ and $a\wedge b:= \min\{a, \:b\}$.

(5) \: For two open sets $U,V$ in $\mathbb{R}^n$,  $V\subset\subset U$ means $\overline{V}\subset U.$

(6) \: For two functions $f,g$, we write $f\lesssim g$ if there exists $C>0$ such that $$ f\leq C g.$$ We write $f\asymp g$ if $f\lesssim g$ and $g\lesssim f$.
\subsection{The Kobayashi metric}

Given a domain $
\Omega \subset \mathbb{C}^{n}\: (n\geq 2)$, the (infinitesimal)
Kobayashi metric is the pseudo-Finsler metric defined by
$$k_{
\Omega}(x ; v)=\inf \left\{|\xi| : f \in \operatorname{Hol}(\mathbb{D}, \Omega), \:\text { with } f(0)=x,
f'(0)\cdot\xi=v\right\}.$$ Define the Kobayashi length of any curve $\sigma:[a,b]\rightarrow \Omega$
to be
$$L_k(\sigma)=\int_{a}^{b} k_{\Omega}\left(\sigma(t) ; \sigma^{\prime}(t)\right) d
t.$$
It is a consequence of a result due to Venturini \cite{VenturiniPseudodistances}, which is based on an observation by Royden \cite{royden1971remarks}, that
the Kobayashi pseudodistance can be given by:
\begin{align*}
K_{\Omega}(x, y)&=\inf_\sigma \big\{L_k(\sigma)| \:\sigma :[a, b]
\rightarrow \Omega \text { absolutely continuous curve }\\
& \text { with } \sigma(a)=x \text { and } \sigma(b)=y \big\}.
\end{align*}
The main property of the Kobayashi pseudodistance is that it is contracted by holomorphic maps: if $f: \Omega_1 \rightarrow \Omega_2$ is a holomorphic map, then
$$
\forall z, w \in \Omega_1 \quad {K}_{\Omega_2}(f(z), f(w)) \leqslant {K}_{\Omega_1}(z, w).
$$
In particular, the Kobayashi distance is invariant under biholomorphisms, and decreases under inclusions: if $\Omega_{1} \subset \Omega_{2} \subset \subset \mathbb{C}^{n}$ are two bounded domains, then we have $K_{\Omega_{2}}(z, w) \leqslant K_{\Omega_{1}}(z, w)$ for all $z, w \in \Omega_{1}$.

 We have some estimates of the Kobayashi metric on convex domains:

 \begin{lemma}[ Theorem 5,\cite{Graham1975Boundary}]\label{est1}
 Suppose $\Omega\subset\mathbb{C}^n$ is a convex domain. If $z\in\Omega$ and $v\in\mathbb{C}^n$ is nonzero, then
 $$\frac{|v|}{2\delta_{\Omega}(z,v)}\leq k_{\Omega}(z,v)\leq \frac{|v|}{\delta_{\Omega}(z,v)}$$
 \end{lemma}
 We also have a global version of the estimate:
 \begin{lemma}[ Proposition 2.5,\cite{forstneric1993proper}]\label{est2}
 Suppose $\Omega\subset\mathbb{C}^n$ is a bounded domain with $C^{1,\alpha}(\alpha>0)$ smooth boundary, then there exists $C>0$ such that for any $x,y\in\Omega$,
 $$K_{\Omega}(x,y)\leq \frac{1}{2}\log\left(1+\frac{|x-y|}{\delta_{\Omega}(x)}\right)+\frac{1}{2}\log\left(1+\frac{|x-y|}{\delta_{\Omega}(y)}\right)+C$$
 \end{lemma}
 \begin{lemma}[Proposition 2, \cite{Nikolov2015The}]\label{est3}
  Suppose $\Omega\subset\mathbb{C}^n$ is a convex domain. If $x,y\in\Omega$, then
  $$K_{\Omega}(x,y)\geq \frac{1}{2}\log\left(1+\frac{|x-y|}{\delta_{\Omega}(x,x-y)\wedge\delta_{\Omega}(y,x-y)}\right)$$
 \end{lemma}
 \begin{lemma}[Proposition 2.4,\cite{MercerComplex}]\label{est4}
 Suppose $\Omega\subset\mathbb{C}^n$ is a convex domain. If $x,y\in\Omega$, then
 $$K_{\Omega}(x,y)\geq\frac{1}{2}\left|\log\frac{\delta_{\Omega}(x)}{\delta_{\Omega}(y)}\right|$$
 \end{lemma}
 \begin{rmk}
 The above estimates suffice for our purposes, but there are some improvements, for instance \cite{ZimmerCharacterizing,Nikolov2015Estimates}.
 \end{rmk}

\section{Convex domains with smooth boundary of finite type }
 Let $\Omega$ be smoothly bounded and convex in some neighborhood of $p\in\partial \Omega$. We suppose $p$ is of finite type (in the sense of D'Angelo), which means that the maximum order of contact of one-dimensional complex analytic varieties (or equivalently one-dimensional complex lines in this case\cite{1992Convex}) with $\partial \Omega$ at $p$, is bounded.
 \subsection{Construction of the polydisks}
Now we will introduce the polydisks constructed by Chen and McNeal (the original construction needs a little correction, see \cite{Zimmer2016Gromov,Nikolov2013On,2004Extremal}).
For a neighborhood $U$ of $p$, let $r$ be the defining function such that $\Omega\cap U=\{z\in U:r(z)<0\}$. By a rotation of the canonical coordinates we can arrange that the normal direction to $\partial \Omega$ at $p$ is given by the $\Re z_{1}$-axis. Then by using the implicit function theorem, we obtain a local defining function of the form $r\left(z_{1}, \ldots, z_{n}\right)=\Re z_{1}-F\left(\Im z_{1}, \ldots, \Re z_{n}, \Im z_{n}\right)$,
where $F$ is a convex function. For $q \in U$ and $\epsilon>0$, we will consider the level sets
$$\partial \Omega_{q, \epsilon}=\{z \in U : r(z)=\epsilon+r(q)\}$$
the defining function has been chosen so that the level sets are convex.

After perhaps shrinking $U$, for every $q \in \Omega\cap U$ and a sufficiently small $\epsilon>0$, we can assign coordinates $\left(z_{1}, \ldots, z_{n}\right)$, $z_{i}=x_{i}+i x_{n+i}$, centered at $q,$ obtained by translating and rotating the canonical coordinates, and numbers $\tau^{i}(q, \epsilon)$ which measure the distance from $q$ to $\partial \Omega_{q, \epsilon}$ along the complex line determined by the $z_{i}$-axis:

First, choose $z_{1}$ so that $dist\left(q, \: \partial \Omega_{q, \epsilon}\right)$ is achieved along the positive $x_{1}$-axis. Let $q_{1,\epsilon}$ be the point in $\partial \Omega_{q, \epsilon}$ such that $\tau_{1}(q,\epsilon)=|q-q_{1,\epsilon}|= dist\left(q, \partial \Omega_{q, \:\epsilon}\right)$, and let $e_1$ be the unit vector in the direction of $x_1$-axis. Next choose a
unit vector $e_2$ in the orthogonal complement of the space $e_1$ (the complex
linear span of $e_1$) such that the minimum distance from $q$ to $\partial \Omega_{q, \epsilon}$ along
directions orthogonal to $e_1$ is achieved along the line given by $e_2$ in a
point $q_{2,\epsilon}$. Choose $z_2$ such that $x_2$-axis lies in the direction of $e_2$ and $\tau_{2}(q,\epsilon)=|q-q_{2,\epsilon}|$. Now continue by choosing a unit vector $e_3$ in the orthogonal complement of $e_1,e_2$, until the basis is complete.
 Also note that the remaining coordinates $z_{i}, \: i=3, \ldots, n$, have the property that the distance from $q$ to $\partial \Omega_{q, \epsilon}$ with respect to the complex direction $z_i$ is achieved on the positive $x_{i}$-axis.

Therefore
$$
P(q,\:\epsilon)=\{|z_i|\leq\tau_{i}(q,\epsilon), \:i=1,\cdots, n\}
$$
is the corresponding polydisk constructed in terms of the minimal basis in $D\cap U$.
\begin{proposition}[Proposition 2.2,\cite{1994Estimates}]\label{del1}
Let $v$ be a unit vector and $v=\sum\limits_{i=1}^{n}a_i e_i$, $a_i\in\mathbb{C}$ and $e_i$ are orthogonal unit vectors determined by the coordinate directions constructed above. For small $\epsilon>0$, we have
\begin{align}
\left(\sum_{i=1}^{n}|a_i|\tau_i(q,\epsilon)^{-1}\right)^{-1}\lesssim\delta_{\Omega_{q,\epsilon}}(q,v)\lesssim\left(\sum_{i=1}^{n}|a_i|\tau_i(q,\epsilon)^{-1}\right)^{-1}
\end{align}
where the constant is independent of $q,\epsilon$ and the vector $v$.
\end{proposition}
\begin{proposition}[Proposition 2.5,\cite{1994Estimates}]\label{doubling}
Given $c>0$, for small $\epsilon>0$ $$\tau_{1}(q,\epsilon)\asymp\epsilon.$$
For $2\leq i\leq n$, if $c>1$
$$c^{1/L}\tau_{i}(q,\epsilon)\lesssim\tau_i(q,c\epsilon)\lesssim c^{1/2}\tau_{i}(q,\epsilon)$$
if $0<c\leq 1$
$$c^{1/2}\tau_{i}(q,\epsilon)\lesssim\tau_i(q,c\epsilon)\lesssim c^{1/L}\tau_{i}(q,\epsilon)$$
where the constants independent of $q,c$ and $\epsilon$.
\end{proposition}
\subsection{The local pseudodistance}
Suppose $x,y\in\Omega\cap U$ and each $q\in\partial\Omega\cap U$ is of finite type at most $L$. Define
$$M(x,y)=\inf\{\epsilon>0:y\in P(x,\epsilon)\},$$
where $P(x,\epsilon)$ is constructed from the coordinates about $x$.
Let $\delta=M(x,y)$ and construct the coordinates about $x$ to the level set $\{r=r(x)+\delta\}$.
For each $2\leq i\leq n$, apply Taylor's Theorem and we have
$$r\circ z(0,\cdots,t_i,\cdots,0)=r(x)+\sum\limits_{k=2}^{L}a_{ik}(x)t_{i}^{k}+\mathcal{O}(|t_i|^{L+1})$$
From P.134 in \cite{1994Estimates},
\begin{equation}\label{pseu}
M(x,y)\asymp|x_1-y_1|+\sum\limits_{i=2}^{n}\sum\limits_{k=2}^{L}|a_{ik}(x)||x_i-y_i|^k.\end{equation}
We provide a proof for completeness.
On one hand, by the definition of $M(x,y)$,  applying Taylor's Theorem in a single variable $2\leq i\leq n$ and 
we have
\begin{align*}
r(x)+\delta=r(x)+\sum\limits_{k=2}^{L}a_{ik}(x)\tau_{i}^{k}(x,\delta)+\mathcal{O}(|\tau_i(x,\delta)|^{L+1}).
\end{align*}
Then, there exists $2\leq k\leq L$ such that $a_{ik}(x)\neq 0$ since $\partial\Omega \cap U$ is of finite type $L$. By a result in \cite{1988Convex}(or see p.113-114 in \cite{1994Estimates}),
\begin{align*}
\delta&=\sum\limits_{k=2}^{L}a_{ik}(x)\tau_{i}^{k}(x,\delta)+\mathcal{O}(|\tau_i(x,\delta)|^{L+1})\\
&\gtrsim\sum\limits_{k=2}^{L}|a_{ik}(x)|\tau_{i}^{k}(x,\delta)
\end{align*}
so for each $2\leq i\leq n$,
$$\delta\gtrsim\sum\limits_{k=2}^{L}|a_{ik}(x)||x_i-y_i|^k.$$
For $i=1$, it is obvious that $\delta\gtrsim|x_1-y_1|.$ Thus,
$$M(x,y)=\delta\gtrsim |x_1-y_1|+\sum\limits_{i=2}^{n}\sum\limits_{k=2}^{L}|a_{ik}(x)||x_i-y_i|^k$$
On the other hand, by the definition of $M(x,y)$ and the continuity of $r$, we have $y\in \partial P(x,\delta)$. Thus, there exists $2\leq i\leq n$ such that
$$|x_i-y_i|=\tau_{i}(x,\delta)$$
and
$$\delta\lesssim\sum\limits_{k=2}^{L}a_{ik}(x)|x_i-y_i|^k$$
or $$|x_1-y_1|=\tau_1(x,\delta)\asymp\delta.$$
Thus $$\delta\lesssim|x_1-y_1|+\sum\limits_{i=2}^{n}\sum\limits_{k=2}^{L}|a_{ik}(x)||x_i-y_i|^k$$
which completes the proof.
\begin{proposition}[Proposition 5.1,\cite{1994Estimates}]\label{pseudodis}
$M(\cdot,\cdot)$ defines a local pseudodistance; i.e. there exists $C>0$ such that for any $x,y,z\in\Omega\cap U$, then
\begin{enumerate}
\item
$M(x,y)=0$ if and only if $x=y$;
\item
$M(x,y)\leq C M(y,x)$;
\item
$M(x,y)\leq C\left(M(x,z)+M(z,y)\right).$
\end{enumerate}
\end{proposition}
Now we prove a Lemma which plays an important role in the proof of the main theorem. With this Lemma we can improve the estimates in \cite{2005Estimation}.
\begin{lemma}\label{pseudodistance}
There exists $\epsilon_0 >0$, such that for any $\epsilon\leq\epsilon_0,k\in\mathbb{N}_{+}$, and
$x,x_1,x_2$ ,$\cdots,x_k,y\in\Omega\cap U$, we have
\begin{align}\label{quasime}
M^{\epsilon}(x,y)\leq 2(M^{\epsilon}(x,x_1)+M^{\epsilon}(x_1,x_2)+\cdots+M^{\epsilon}(x_k,y))
\end{align}
\end{lemma}
\begin{rmk}
 Substitute $M(x,y)$ with $e^{(x|y)_{o}}$ for some fixed point $o\in \Omega\cap U.$ This is the well known estimate of the visual metrics on the Gromov boundary, see \cite{1990Sur}, Chaper 7. Actually one can adapt the proof there to this Lemma. Moreover in Section 5 we will prove that $M(x,y)\asymp e^{(x|y)_{\omega}}$ locally.
\end{rmk}
\begin{proof}
By Proposition \ref{pseudodis}, there is $C>1$ such that for any $x,y,z\in \Omega\cap U$, we have
\begin{align}\label{qausimetric}
M(x,y)\leq C(M(x,z)+M(z,y))\leq 2C(M(x,z)\vee M(z,y))
\end{align}
We claim that $\epsilon_0=\frac{\log2}{2\log2C}$ works. Let $\epsilon\leq \frac{\log2}{2\log2C}$, or equivalently $(2C)^{2\epsilon}\leq 2.$
We prove by induction on the integer $k$.

Obviously the inequality (\ref{quasime}) holds when $k=1$. Suppose the inequality holds for $k\leq d-1$. Now we assume that $k=d$.

Put $R:=\sum\limits_{i=0}^{d}M^{\epsilon}(x_i,x_{i+1})$. Let $p$ be the biggest integer $q$ such that
$\sum\limits_{i=0}^{q}M^{\epsilon}(x_i,x_{i+1})\leq\frac{R}{2}.$
Then we have
$$\sum\limits_{i=0}^{p}M^{\epsilon}(x_i,x_{i+1})\leq\frac{R}{2}$$
and
$$\sum\limits_{i=p+2}^{d}M^{\epsilon}(x_i,x_{i+1})\leq\frac{R}{2}$$
By the assumption,
$$M^{\epsilon}(x,x_{p+1})\leq 2\left(\sum\limits_{i=0}^{p}M^{\epsilon}(x_i,x_{i+1})\right)\leq R$$
and
$$M^{\epsilon}(x_{p+2},y)\leq 2\left(\sum\limits_{i=p+2}^{d}M^{\epsilon}(x_i,x_{i+1})\right)\leq R$$
also, $M^{\epsilon}(x_{p+1},x_{p+2})\leq R$. Then by (\ref{qausimetric}),
$$M(x,y)\leq (2C)^2 (M(x,x_{p+1})\vee M(x_{p+1},x_{p+2})\vee M(x_{p+2},x_d))$$
Thus we have
$$M^{\epsilon}(x,y)\leq (2C)^{2\epsilon}R\leq 2R=2\sum\limits_{i=0}^{d}M^{\epsilon}(x_i,x_{i+1})$$
which completes the proof.
\end{proof}

\section{Estimates of the Kobayashi distance}
Before we prove Theorem \ref{main}, we need some preparations. The first statement is a direct generalization of a result in \cite{balogh2000gromov}.
\begin{lemma}\label{dist}
Suppose that $\Omega=\{x\in\mathbb{R}^{n}: r(x)<0\},n\geq 2$, is a bounded domain with $C^2$-smooth boundary. Then there exists $\delta_0 >0$ such that

(a). For every point $x\in N_{\delta_0}(\partial\Omega)=\{z\in\Omega: |\delta_{\Omega}(z)|<\delta_0\}$, there exists a unique point $\pi(x)\in\partial\Omega$ with $|x-\pi(x)|=\delta_{\Omega}(x).$

(b). For the fibers of the map $\pi:N_{\delta_0}(\partial\Omega)\rightarrow\partial\Omega$ we have
$$\pi^{-1}(p)=(p-\delta_{0}\vec{n}(p),p+\delta_{0}\vec{n}(p)),$$
where $\vec{n}(p)$ is the outer unit normal vector of $\partial\Omega$ at $p\in\partial\Omega.$

(c). The gradient of the defining function $r$ satisfies
$$\nabla r(x)=|\nabla r(x)|n(\pi(x))$$
for all $x\in N_{\delta_0}(\partial\Omega)$.

(d). $$|r(x)|\asymp\delta_{\Omega}(x)$$
for all $x\in N_{\delta_0}(\partial\Omega)$.
\end{lemma}
\begin{rmk}\label{d1}
	By (d), there exists $\delta_1>0$ such that for any $x\in \Omega$ satisfying $|r(x)|\leq \delta_1,$ then $x\in N_{\delta_0}(\partial\Omega)$.
\end{rmk} 
\begin{proof}
The proof of (a) and (b) can be found in \cite{1959Curvatures,balogh2000gromov}. For (c),
suppose that $\pi(x)=p.$ Denote $x=p-t_0 \vec{n}(p)$ and $x_t=p-(t_0+t)\vec{n}(p)$ for some fixed $t_0 >0$ and $t>0$ small enough, then by Taylor expansion
$$r(x_t)-r(x)=(\nabla r(x)\cdot \vec{n}(p))t+o(t)$$
which means $\nabla r(x)\cdot \vec{n}(p)>0$.
Let $\partial\Omega_{x,0}=\{z\in \mathbb{R}^n:r(z)=r(x)\}$. Let $\gamma:[-\epsilon,\epsilon]\rightarrow \partial\Omega_{x,0}$ be a $C^{1}$-smooth curve such that $\gamma(0)=x.$ Since $r(\gamma(s))\equiv r(x)$ for any $s\in[-\epsilon,\epsilon]$, we have $$\nabla r(x)\cdot \gamma'(0)=0$$
which means $\nabla r(x)\in \mathbb{R}\cdot \vec{n}(p)$ and completes the proof.

Since $\nabla r\neq 0$ on $\partial\Omega$, (d) follows from (c).
\end{proof}

Let $b=(b_1,\cdots,b_n)$ where $b_i>0 (i=1,\cdots,n)$ and $x\in\mathbb{C}^n$. Denote the polydisk
$$D(x,b)=\{z\in\mathbb{C}^n:|z_i-x_i|<b_i,1\leq i\leq n\}$$ Later we will use this definition of $D(x,b)$ with respect to the coordinates centered at $x$ constructed in 3.1.
\begin{lemma}\label{del2}
For any point $y\in \partial D(x,b)$  we have
$$\left(\sum_{i=1}^{n}|a_i|b_{i}^{-1}\right)^{-1}\leq|x-y|\leq n\left(\sum_{i=1}^{n}|a_i|b_{i}^{-1}\right)^{-1}$$
where $a_i=\frac{x_i-y_i}{|x-y|}\in\mathbb{C}$.
\end{lemma}
\begin{proof}
Since $y\in \partial D(x,b)$, there exists $1\leq j\leq n$ such that $|x_j-y_j|=b_j$. Thus
$$\sum_{i=1}^{n}|a_i|b_{i}^{-1}\geq |a_j|b_{j}^{-1}=\frac{1}{|x-y|}$$
and the inequality on the left hand side follows.

On the other hand, for each $1\leq i\leq n$, $$|x_i-y_i|\leq b_i$$
thus
$$\sum_{i=1}^{n}|a_i|b_{i}^{-1}\leq \sum_{i=1}^{n}|x-y|^{-1}=n|x-y|^{-1}$$
and the inequality on the right hand side follows.
\end{proof}
Let $\Omega=\{r<0\}\subset\mathbb{C}^n$ be smoothly bounded and convex in some neighborhood of $p\in\partial\Omega$, and suppose that $p$ is a point of finite type in sense of D'Angelo. Thus there exists a neighbourhood $U$ such that for any point $x\in\Omega\cap U,$ the polydisks defined in Section 3 can be constructed.
\begin{lemma}\label{polyest}
Given $c>0$, there exists $c_0>0$ depending only on $c$ such that for any point $x\in\Omega\cap U$ and $y\in P(x,c|r(x)|)\cap \Omega$, we have
$$K_{\Omega}(x,y)\geq c_0.$$
\end{lemma}
\begin{proof}
Denote $v=\frac{x-y}{|x-y|}$ and $a_{i}=\frac{x_i-y_i}{|x-y|}$ for each $1\leq i\leq n$. Then, we have 
$$v=\frac{x-y}{|x-y|}=\sum_{i=1}^{n}a_ie_i.$$By Lemma \ref{est3}, Proposition \ref{del1} and Lemma \ref{del2}, we have
\begin{align*}
K_{\Omega}(x,y)&\geq \frac{1}{2}\log\left(1+\frac{|x-y|}{\delta_{\Omega}(x,v)}\right)\gtrsim \log\left(1+\frac{\left(\sum\limits_{i=1}^{n}a_{i}\tau_{i}(x,c|r(x)|)^{-1}\right)^{-1}}{\left(\sum\limits_{i=1}^{n}a_{i}\tau_{i}(x,|r(x)|)^{-1}\right)^{-1}}\right).\\
\end{align*}
Then by Proposition \ref{doubling}, we have
$$K_{\Omega}(x,y)\geq c_0$$
 where $c_0$ depends only on $c$.
\end{proof}
The following localization of the Kobayashi metric can simplify the proof of the Theorem \ref{main}.
\begin{lemma}[Theorem 16.3,\cite{Zimmer2016Gromov}]\label{local}
Let $\Omega$ be a bounded domain in $\mathbb{C}^n$, and suppose $p\in\partial\Omega$ is a smooth boundary point of finite type and there exists a open set $ U$ of $p$ such that $\Omega\cap U$ is convexifiable, then there exists a open set $V$, $p\in V\subset\subset U$ and a constant $C>0$ such that for any $x,y\in\Omega\cap V$
$$K_{\Omega\cap U}(x,y)-C\leq K_{\Omega}(x,y)\leq K_{\Omega\cap U}(x,y)+C$$
\end{lemma}
\begin{rmk}
For our purposes the above result suffices, but there are more results about the localization, for instance\cite{2020Localization,2021Visibility,2021Quantitative}.
\end{rmk}

Now, we can give a proof of Theorem \ref{main}, by Lemma \ref{local}, we can assume that $\Omega$ is convex globally. In what follows we will denote by $C$ positive constants only depending on the various constants associated with $\Omega$ and $U$ in Section 2 and 3. The actual value of $C$ does not matter and may change even within the same line. Define $$g_1(x,y)=\log\left[\frac{M(x,y)+|r(x)|\vee|r(y)|}{\sqrt{|r(x)r(y)|}}\right].$$
By (d) of Lemma \ref{dist}, $g_1(x,y)-C\leq g(x,y)\leq g_1(x,y)+C$. Thus we only need to prove the following:
\begin{thm}
Let $\Omega=\{r<0\}\subset\mathbb{C}^n$ be a smoothly bounded convex domain, suppose that $\xi\in\partial \Omega$ is a point of finite type in sense of D'Angelo. Then there exists a neighborhood $U$ of $\xi\in\partial\Omega$ and constants $C>0$ such that for any $x,y\in\Omega\cap U$,
$$g_1(x,y)-C\leq K_{\Omega}(x,y)\leq g_1(x,y)+C.
$$
\end{thm}

\begin{proof}

We choose $U$ small enough such that the polydisks and the pseudodistance $M$ are well defined on $\Omega\cap U,$ and for any $x,y\in\Omega\cap U$, $|r(x)|<\delta_1$(Remark  \ref{d1}) and $M(x,y)<\delta_1.$
\bigskip

First we estimate the upper bounds of $K_{\Omega}$.
\bigskip

Case 1. $M(x,y)\leq |r(x)|\vee|r(y)|$.

We may assume $|r(y)|>|r(x)|$. Then
$$g_1(x,y)=\frac{1}{2}\log\frac{|r(y)|}{|r(x)|}+ C.$$
Let $x'\in x+\mathbb{R}\cdot \vec{n}_{\pi(x)}$ such that $r(x')=r(y)$. By Lemma \ref{est2} and Lemma \ref{dist},
$$K_{\Omega}(x,x')\leq\frac{1}{2}\log\frac{|r(y)|}{|r(x)|}+C.$$
It remains to find an upper bound for $K_{\Omega}(x',y)$.

Let $\overline{ab}$ denote the straight line segment joining $a$ and $b$. By the convexity of $\Omega$, for any $z\in \overline{x'y}$, we claim that
$$\delta_{\Omega}(z,x'-y)\geq \delta_{\Omega}(x',x'-y)\wedge \delta_{\Omega}(y,x'-y).$$
Let $H:=z+(x'-y)\cdot\mathbb{C}$, by definition 
$$\delta_{\Omega}(z,x'-y)=\delta_{\Omega\cap H}(z).$$
Similarly, we have $$\delta_{\Omega}(x',x'-y)=\delta_{\Omega\cap H}(x')$$ and $$\delta_{\Omega}(y,x'-y)=\delta_{\Omega\cap H}(y).$$
Note that $\Omega\cap H$ is also convex and $z\in \overline{x'y}$, so
$$\delta_{\Omega\cap H}(z)\geq \delta_{\Omega\cap H}(x')\wedge\delta_{\Omega\cap H}(y)$$
which proves the claim.

Suppose $\delta_{\Omega}(x',x'-y)\leq\delta_{\Omega}(y,x'-y)$.
By Lemma \ref{dist}, we know $x'$ lies in the $x_1$-axis of the coordinates centered at $x$ and $\nabla r(x')$ is parallel to $\nabla r(x)$. Then we can deduce that $|x_{1}(x')|\leq |x_{1}(y)|$ by the convexity of $\partial\Omega_{x',0}=\{z\in U:r(z)=r(x')\}$ and the fact that $y\in\partial\Omega_{x',0}$. Thus $x'\in P(x,M(x,y))$ which means $M(x,x')\leq M(x,y)$. Then 
$$M(x',y)\leq C(M(x,y)+M(x,x'))\leq C|r(y)|$$
Thus $y\in P(x',C|r(y)|)$. Then by Lemma \ref{est1}, Proposition \ref{del1} and the proof of Lemma \ref{del2}, we have
$$K_{\Omega}(x',y)\leq L_{k}(\overline{x'y})\lesssim \frac{|x'-y|}{\delta_{\Omega}(x',x'-y)}\lesssim \frac{\left(\sum\limits_{i=1}^{n}a_{i}\tau_{i}(x',C|r(x')|)^{-1}\right)^{-1}}{\left(\sum\limits_{i=1}^{n}a_{i}\tau_{i}(x',|r(x')|)^{-1}\right)^{-1}}\lesssim C^{1/2}$$
where $a_i=\frac{x'_{i}-y_i}{|x'_{i}-y_i|}$, $1\leq i\leq n$.

If $\delta_{\Omega}(y,x'-y)<\delta_{\Omega}(x',x'-y)$, since $M(y,x')\leq CM(x',y)$ and $r(x')=r(y)$, similarly we can get the estimate and so is the case when  $|r(y)|\leq |r(x)|$.
\bigskip

Case 2. $M(x,y)>|r(x)|\vee|r(y)|$.

Then $$g_1(x,y)=\log\frac{M(x,y)}{\sqrt{r(x)r(y)}}+C.$$Let $x'\in x+\mathbb{R}\cdot \vec{n}_{\pi(x)}$ such that $|r(x')|=M(x,y)$ and $y'\in y+\mathbb{R}\cdot \vec{n}_{\pi(y)}$ such that $|r(y')|=M(y,x)$.
Then
$$K_{\Omega}(x,x')+K_{\Omega}(y,y')\leq\log\frac{M(x,y)}{\sqrt{r(x)r(y)}}+C$$
If  $\delta_{\Omega}(x',x'-y')\leq\delta_{\Omega}(y',x'-y')$, by the convexity of $\Omega$, for any $z\in \overline{x'y'}$, we have
$$\delta_{\Omega}(z,x'-y')\geq \delta_{\Omega}(x',x'-y').$$
Note that $M(x,x')\leq M(x,y)$, and $M(y,y')\leq M(y,x)$ then
\begin{align*}
M(x',y')\leq C^2(M(x',x)+M(x,y)+M(y,y'))\leq CM(x,y)
\end{align*}
Thus $y'\in P(x',CM(x,y))$, and by Lemma \ref{est1}, Proposition \ref{del1} and the proof of Lemma \ref{del2} we estimate
$$K_{\Omega}(x',y')\leq L_{k}(\overline{x'y'})\lesssim \frac{|x'-y'|}{\delta_{\Omega}(x',x'-y')}\lesssim \frac{\left(\sum\limits_{i=1}^{n}a_{i}\tau_{i}(x',C|r(x')|)^{-1}\right)^{-1}}{\left(\sum\limits_{i=1}^{n}a_{i}\tau_{i}(x',|r(x')|)^{-1}\right)^{-1}}\lesssim C^{1/2}$$
If $\delta_{\Omega}(y',x'-y')<\delta_{\Omega}(x',x'-y')$, similarly we can get the estimate.
 Thus we have
$$K_{\Omega}(x,y)\leq g_1(x,y)+C.$$

\bigskip
Then we estimate the lower bounds of $K_{\Omega}$. We may assume $|r(y)|>|r(x)|$.
\bigskip

Case 1. $M(x,y)\leq |r(x)|\vee|r(y)|.$

We have $$g_1(x,y)=\frac{1}{2}\log \frac{|r(y)|}{|r(x)|}+C.$$ By Lemma \ref{est4} we have $$K_{\Omega}(x,y)\geq\frac{1}{2}\log\frac{|r(y)|}{|r(x)|}-C$$
which completes the proof in this case.

Case 2. $M(x,y)>|r(x)|\vee|r(y)|$.

We have $$g_1(x,y)=\log \frac{M(x,y)}{\sqrt{|r(x)r(y)|}}+C.$$ Let $\gamma:[0,1]\rightarrow \Omega$ be arbitrary curve joining $x$ and $y$. Define $H=\max\limits_{z\in\gamma}|r(z)|$. There exists $t_0\in[0,1]$ such that $|r(\gamma(t_0)|=H$. Consider two subcurves $\gamma_1=\gamma|[0,t_0]$ and $\gamma_2=\gamma|[t_0,1]$. There are two possibilities.

If $H\geq M(x,y)$, then by Lemma \ref{est4} $$L_{k}(\gamma_1)\geq \frac{1}{2}\log\frac{H}{|r(x)|}-C$$
and
$$L_{k}(\gamma_2)\geq \frac{1}{2}\log\frac{H}{|r(y)|}-C$$
Thus, $$L_{k}(\gamma)\geq \log\frac{H}{\sqrt{r(x)r(y)}}-C\geq\log\frac{M(x,y)}{\sqrt{|r(x)r(y)|}}-C,$$
which completes the proof.

The other possibility is when $H<M(x,y)$. Since $|r(x)|\leq H$, there exists $k\in \mathbb{N}_{+}$ such that
$$\frac{H}{2^{k}}<|r(x)|\leq\frac{H}{2^{k-1}}$$
Consider the curve $\gamma_1$ and define $0=s_0\leq s_1<\cdots<s_k\leq t_0$ as follows. Let
$$s_j=\max\{s\in[0,t_0]:|r(\gamma(s))|=\frac{H}{2^{k-j}}\}$$
for $j=1,\cdots,k.$ Put $x_j=\gamma(s_j)$ for $j=0,\cdots,k.$ Note that $$1\leq\frac{|r(x_j)|}{|r(x_{j-1})|}\leq 2$$ for $j=1,\cdots,k.$

 Then let $\epsilon =\epsilon_0$ in Lemma \ref{pseudodistance} and $c_1=\left(\frac{1-1/2^{\epsilon}}{16}\right)^{1/\epsilon}$, we consider two alternatives.

 a. First we assume that there exists an index $l\in\{1,\cdots,k\}$ such that
 $$M(x_{l-1},x_l)>c_{1} \frac{M(x,y)}{2^{k-l}}.$$
 Note that for $t\in[s_{l-1},s_l]$ we have
 $$|r(\gamma(t))|\leq\frac{H}{2^{k-l}}.$$
 Then we denote $p=x_{l-1}$ and $q=x_l$. Since
 \begin{align}\label{alternativea}
 M(p,q)>c_{1} \frac{M(x,y)}{2^{k-l}}\geq c_{1}|r(p)|
 \end{align}
 it follows that $$q\notin P(p,c_{1}|r(p)|).$$
 Then $\gamma\cap \partial P(p,c_{1}|r(p)|)\neq\emptyset$. Denote $u_0=s_{l-1}$ and choose $u_0<u_1<s_{l}$ such that $$\gamma(u_1)\in\partial P(p,c_{1}|r(p)|).$$
 Then denote $p_0=p$ and $p_1=\gamma(u_1).$
 This shows the following set is not empty:
 \begin{align*}
 S^{p,q}=\{m\in\mathbb{N}_{+}&|\text{there exist points $u_1,\cdots,u_m$ so that}\quad  s_{l-1}=u_0<u_1<\\&\cdots<u_m\leq s_{l},
 \gamma(u_{\nu})\in\partial P(p_{\nu-1},c_{1}|r(p_{\nu-1})|)), 1\leq\nu\leq m\}.\end{align*}
 Then $S^{p,q}$ is a finite set. Indeed, if $m\in S^{p,q}$, then there exist points $u_1,\cdots,u_m$ satisfying $s_{l-1}=u_0<u_1<\cdots<u_m\leq s_l$ and $\gamma(u_{\nu})\in\partial P(p_{\nu-1},c_{1}|r(p_{\nu-1})|),1\leq\nu\leq m$. By Lemma \ref{polyest}, we have
 \begin{align*}
L_{k}(\gamma|[s_{l-1},s_l])\geq\sum\limits_{\nu=1}^{m}K_{\Omega}(p_{\nu-1},p_{\nu})\geq c_0 m.
 \end{align*}
 Now we define the number $m_1:=\max S^{p,q}$.  We estimate $m_1$ from below using Lemma \ref{pseudodistance}.
 \begin{align*}
 M^{\epsilon}(p,q)&\leq 2\left(\sum\limits_{\nu=1}^{m_1}M^{\epsilon}(p_{\nu-1},p_{\nu})+M^{\epsilon}(p_{m_1},q)\right)\\
 &\leq 2 c_{1}^{\epsilon}\left(\sum\limits_{\nu=1}^{m_1}|r^{\epsilon}(p_{\nu-1})|+|r^{\epsilon}(p_{m_1})|\right)\\
 &\leq 2c_{1}^{\epsilon}(m_1+1)\left(\frac{H}{2^{k-l}}\right)^{\epsilon}.
 \end{align*}
Note that the maximality of $m_1$ is used to get the upper bound on $M^{\epsilon}(p_{m_1},q)$. Thus
$$m_1\geq \frac{1}{2c_{1}^{\epsilon}}\left(\frac{M(p,q)}{\frac{H}{2^{k-l}}}\right)^{\epsilon}-1\geq \frac{1}{2c_{1}^{\epsilon}}\left(\frac{M(x,y)}{H}\right)^{\epsilon}-1.$$
Then by Lemma \ref{polyest} there exsits a constant $c_0=c_0(c_1)$ which depends only on $c_1$ so that 
\begin{align}\label{number}
L_{k}(\gamma|[s_{l-1},s_l])\geq\frac{c_0}{2c_{1}^{\epsilon}}\left(\frac{M(x,y)}{H}\right)^{\epsilon}-c_0.
\end{align}
Let $t_1:=s_k\leq t_0$. As a consequence of Lemma \ref{est4} and (\ref{number}) we have
\begin{align*}
L_{k}(\gamma|[0,t_1])=&L_{k}(\gamma|[0,s_{l-1}])+L_{k}(\gamma|[s_{l-1},s_{l}])+L_{k}(\gamma|[s_{l},s_k])\\
&\geq\frac{1}{2}\log\left(\frac{r(x_{l-1})}{r(x_0)}\right)+\frac{c_0}{2c_{1}^{\epsilon}}\left(\frac{M(x,y)}{H}\right)^{\epsilon}+\frac{1}{2}\log\left(\frac{r(x_{k})}{r(x_l)}\right)-C\\
&\geq \frac{1}{2}\log\left(\frac{r(x_{k})}{r(x_0)}\right)+\frac{c_0}{2c_{1}^{\epsilon}}\left(\frac{M(x,y)}{H}\right)^{\epsilon}-C\\
&\geq\frac{1}{2}\log\left(\frac{H}{|r(x)|}\right)+\frac{c_0}{2c_{1}^{\epsilon}}\left(\frac{M(x,y)}{H}\right)^{\epsilon}-C.
\end{align*}
Note that we have obtained a lower bound (not finished the proof) for alternative a and will get a similar lower bound for alternative b. Then we will prove the final estimate for two alternatives.

 b. The second alternative is that
 $$M(x_{j-1},x_j)\leq c_{1}\frac{M(x,y)}{2^{k-j}}$$
 for all $j=1,\cdots,k.$ This implies
 \begin{align*}
 M^{\epsilon}(x,\gamma(t_1))&\leq 2\sum\limits_{j=1}^{k} M^{\epsilon}(x_{j-1},x_j)\leq \frac{2c_{1}^{\epsilon}}{1-1/2^{\epsilon }}M^{\epsilon}(x,y)\\
 &\leq \frac{1}{8} M^{\epsilon}(x,y).
 \end{align*}
 On the other hand,
 $$L_{k}(\gamma|[0,t_1])\geq \frac{1}{2}\log\frac{H}{|r(x)|}-C.$$
 Summarizing this discussion we obtain the following two possibilities:
 \begin{align}
L_{k}(\gamma|[0,t_1])\geq\frac{1}{2}\log\left(\frac{H}{|r(x)|}\right)+C\left(\frac{M(x,y)}{H}\right)^{\epsilon}-C
\end{align}
or
\begin{align}\label{est5}
 L_{k}(\gamma|[0,t_1])\geq \frac{1}{2}\log\frac{H}{|r(x)|}-C \quad \text{and} \quad M^{\epsilon}(x,\gamma(t_1))\leq \frac{1}{8} M^{\epsilon}(x,y)
 \end{align}
 Applying similar considerations to the curve $\gamma_2$ instead of $\gamma_1$ we find $t_2\in[t_0,1]$ such that one of the following alternatives holds

 \begin{align}
L_{k}(\gamma|[t_2,1])\geq\frac{1}{2}\log\left(\frac{H}{|r(y)|}\right)+C\left(\frac{M(x,y)}{H}\right)^{\epsilon}-C
\end{align}
or
\begin{align}\label{est6}
 L_{k}(\gamma|[t_2,1])\geq \frac{1}{2}\log\frac{H}{|r(y)|}-C \quad \text{and} \quad M^{\epsilon}(y,\gamma(t_2))\leq \frac{1}{8} M^{\epsilon}(x,y)
 \end{align}
 Let us suppose (\ref{est5})(\ref{est6}) hold simultaneously. Then
 \begin{align*}
 M^{\epsilon}(\gamma(t_1),\gamma(t_2))&\geq \frac{1}{2}\left( M^{\epsilon}(x,y)- M^{\epsilon}(x,\gamma(t_1))-M^{\epsilon}(\gamma(t_2),y)\right)\\
 %&\geq M^{\epsilon}(x,y)(\frac{1}{2}-\frac{4c_{1}^{\epsilon}}{1-1/2^{\epsilon}})\\
 &\geq \frac{1}{4}M^{\epsilon}(x,y)\geq\frac{1}{4}H^{\epsilon}
 \end{align*}
 Then $$M(\gamma(t_1),\gamma(t_2))\geq\left(\frac{1}{4}\right)^{1/\epsilon}\cdot H=\left(\frac{1}{4}\right)^{1/\epsilon}\cdot |r(t_1)|.$$ This is analogous to (\ref{alternativea}) in the first alternative , so similarly we can conclude that
 $$L_{k}(\gamma|[t_1,t_2])\geq 2c_0\left(\frac{M(x,y)}{H}\right)^{\epsilon}$$
 Consequently,
 \begin{align*}
 L_{k}(\gamma)=&L_{k}(\gamma|[0,t_1])+L_{k}(\gamma|[t_1,t_2])+L_{k}(\gamma|[t_2,1])\\
 &\geq \log\frac{H}{\sqrt{r(x)r(y)}}+ 2c_0\left(\frac{M(x,y)}{H}\right)^{\epsilon}-C
 \end{align*}
 This inequality also holds for other cases.
 Let $$f(t)=\log\frac{t}{\sqrt{r(x)r(y)}}+ 2c_0\left(\frac{M(x,y)}{t}\right)^{\epsilon}$$
 Cauculus shows that the function $f$ has a minimum if $$t=(2c_0\epsilon)^{1/\epsilon}M(x,y).$$ This gives the lower bound
 $$L_{k}(\gamma)\geq\log\frac{M(x,y)}{\sqrt{r(x)r(y)}}-C $$
 If we take the infimum over all admissible curves $\gamma$, then we have
 $$K_{\Omega}(x,y)\geq g_1(x,y)-C$$
 which completes the proof.
 \end{proof}
\section{Gromov hyperbolicity}
We begin with some necessary definitions and results concerning Gromov hyperbolicity.
\begin{defn}
Let $(X, \:d)$ be a metric space. Given three points $x, \:y, \:o \in$ $X,$ the {\it Gromov product} of $x,\:y$ with respect to $o$ is given by
$$(x | y)_{o}=\frac{1}{2}\Big(d(x, \:o)+d(o, \:y)-d(x, \:y)\Big).$$
A proper geodesic metric space $(X, \:d)$ is called {\it Gromov hyperbolic} (or $\delta$-{hyperbolic}), if there exists $\delta \geq 0$ such that, for all $o, \:x, \:y, \:z
\in X$,
$$(x | y)_{o} \geq \min \left\{(x | z)_{o},(z | y)_{o}\right\}-\delta.$$
\end{defn}
\begin{defn}[Thin triangle property]
Let $(X, \:d)$ be a geodesic metric space. A geodesic triangle has $\delta$-thin property for some $\delta>0$ if each side of a geodesic triangle lies in the $\delta$-neighbourhood of the other two sides.
\end{defn}
\begin{proposition}[\cite{1990Sur}]
Let $(X, \:d)$ be a geodesic metric space. The space $(X, \:d)$ is Gromov hyperbolic if and only if any geodesic triangle has the thin triangle property.
\end{proposition}
%By the triangle inequality, we know that
%$$(x | y)_{o}\leq d(o,[x,y]),$$
%where $[x,\: y]$ is a geodesic connecting $x$ and $y$ in $(X,\: d)$. Moreover, if $X$ is Gromov hyperbolic, then
%\begin{align}
%|(x|y)_{o}-d(o,[x,y])|\leq \delta'
%\end{align}
%for some $\delta'>0$.

Then, we introduce the definition of Gromov boundary, given a fixed point $o$:
\begin{enumerate}
\item
A sequence $\{x_i\}$ in $X$ is called a {\it Gromov sequence} if $(x_i|x_j)_o\rightarrow \infty$ as $i,$ $j\rightarrow \infty.$
\item
Two such sequences $\{x_i\}$ and $\{y_j\}$ are said to be {\it equivalent} if $(x_i|y_i)_o\rightarrow \infty$ as $i\to\infty$.
\item
The {\it Gromov boundary} $\partial_G X$ of $X$ is defined to be the set of all equivalence classes of Gromov sequences, and $\overline{X}^G=X \cup \partial_G X$ is called the {\it Gromov closure} of $X$.
\item
For $a\in X$ and $b\in \partial_G X$, the Gromov product $(a|b)_o$ of $a$ and $b$ is defined by
$$(a|b)_o= \inf \big\{ \liminf_{i\rightarrow \infty}(a|b_i)_o:\; \{b_i\}\in b\big\}.$$
\item
For $a,\: b\in \partial_G X$, the Gromov product $(a|b)_o$ of $a$ and $b$ is defined by
$$(a|b)_o= \inf \big\{ \liminf_{i\rightarrow \infty}(a_i|b_i)_o:\; \{a_i\}\in a\;\;{\rm and}\;\; \{b_i\}\in b\big\}.$$
\end{enumerate}
For a Gromov hyperbolic space $X$, one can define a class of {\it visual metrics} on $\partial_G X$ via the extended Gromov products, see \cite{BMHA,Schramm1999Embeddings}. For any metric $\rho_{G}$ in this class there exist a parameter $\epsilon>0$ and a base point $w \in X$ such that
\begin{align}\label{ee}\rho_{G}(a, b) \asymp \exp \left(-\epsilon(a|b)_{w}\right), \quad \text { for } a, \:b \in \partial_{G} X.
\end{align}

The following Lemma essentially shows that a bounded convex domain with smooth boundary of finite type is "locally" Gromov hyperbolic.
\begin{lemma}\label{gromov2}
Let $\Omega=\{z\in\mathbb{C}^n:r(z)<0\}$ be a bounded convex domain with smooth boundary of finite type, for all $ p\in\partial \Omega$, there exists an open neighbourhood $U$ of $p$ and $C>0$ such that for any $x,y,z,\omega\in U\cap \Omega$,
\begin{align}\label{gromov1}
(x|y)_{\omega}\geq(x,z)_{\omega}\wedge(z,y)_{\omega}-C.\end{align}
\end{lemma}
\begin{proof}
First, for any point $p\in\partial\Omega$, by Theorem \ref{main} there exists a open neighbourhood $U$ such that the estimates hold for every points in $\Omega\cap U.$ Similar to the proof in \cite{balogh2000gromov}, we will prove the following equivalent statement: for all $x_{i}(i=1,2,3,4)\in\Omega\cap U,$ there holds
\begin{align}\label{gromov}
 \nonumber K_{\Omega}(x_1,x_2)&+K_{\Omega}(x_3,x_4)\\
 &\leq (K_{\Omega}(x_1,x_3)+K_{\Omega}(x_2,x_4))\vee (K_{\Omega}(x_1,x_4)+K_{\Omega}(x_2,x_3))+C.
 \end{align}
 Suppose we are given numbers $r_{ij}\geq 0$ such that $r_{ij}\leq Cr_{ji}$ and $r_{ij}\leq C(r_{ik}+r_{kj})$ for $i,j,k\in\{1,2,3,4\}.$ Then we claim that
 $r_{12}r_{34}\leq 4C^4 (r_{13}r_{24}\vee r_{14}r_{23})$.
 We may assume that $r_{13}$ is the smallest of the quantities $r_{ij}$ on the right hand side of the inequality. Then
 $$r_{12}\leq C(r_{13}+r_{32})\leq C(C+1)r_{23}\leq 2C^2 r_{23}$$ and $r_{34}\leq C(r_{31}+r_{14})\leq 2C^2 r_{14}$, then the inequality follows.

 Now let $x_i,i\in\{1,2,3,4\}$, be four arbitrary points in $\Omega\cap U$, and $h_i=|r(x_i)|$. Set $d_{ij}=M(x_i,x_j)$ and $r_{ij}=d_{ij}+h_i\vee h_j.$ Then
 \begin{align*}
 (d_{12}+&h_1\vee h_2)(d_{34}+h_3\vee h_4)\leq\\
 & 4C^4 (d_{13}+h_1\vee h_3)(d_{24}+h_2\vee h_4)\vee (d_{14}+h_1\vee h_4)(d_{24}+h_2\vee h_3)
  \end{align*}
  By Theorem \ref{main} the above display translates to (\ref{gromov}).
\end{proof}
\begin{lemma}\label{gromov3}
Let $\Omega=\{z\in\mathbb{C}^n:r(z)<0\}$ be a bounded convex domain with smooth boundary of finite type,  and let $p\in\partial \Omega$. Then there exists an open neighbourhood $U$ of $p$ and $\delta>0$ such that for any $ x,y,z\in\Omega\cap U$, the geodesic triangle $\triangle xyz$ has the $\delta$-thin property.
\end{lemma}
\begin{proof}
Without loss of generality, we only prove $[x,y]$ lies in the $\delta$-neighbourhood of the other sides. By Lemma \ref{gromov2}, we know that (\ref{gromov1}) holds in $\Omega\cap U$.
Since
  $$K_{\Omega}(x,y)=(y|z)_{x}+(x|z)_{y}$$
  for all $ u\in[x,y]$, we have $K_{\Omega}(x,u)\leq(y|z)_{x}$ or $K_{\Omega}(y,u)\leq (x|z)_{y}.$
  Suppose $K_{\Omega}(x,u)\leq(y|z)_{x}$, by the triangle inequality $(y|z)_{x}\leq K_{\Omega}(x,z)$, thus we can choose a point $v\in[x,z]$ such that $K_{\Omega}(x,v)=K_{\Omega}(x,u)$. Then
$$(u|v)_{x}=\frac{1}{2}(K_{\Omega}(x,u)+K_{\Omega}(x,v)-K_{\Omega}(u,v))=K_{\Omega}(x,u)-\frac{1}{2}K_{\Omega}(u,v).$$
  On the other hand, by (\ref{gromov1})
  \begin{align*}
  (u|v)_{x}&\geq (u|y)_{x}\wedge(y|z)_{x}\wedge(z|v)_{x}-2C\\
  &\geq K_{\Omega}(x,u)-2C.
  \end{align*}
  Thus,
  $$K_{\Omega}(u,v)\leq 4C.$$
  Similarly, if we suppose $K_{\Omega}(y,u)\leq(x|z)_{y}$, then we can find an analogous $v\in[y,z]$ and conclude $K_{\Omega}(u,v)\leq 4C.$
  Thus, let $\delta= 4C$ and we complete the proof.
\end{proof}
\begin{lemma}\label{visibility}
Let $\Omega=\{z\in\mathbb{C}^n:r(z)<0\}$ be a bounded convex domain with smooth boundary of finite type  $L$, let $p\in\partial \Omega$, and let $U$ be the open neighborhood of $p$ where the polydisks and local pseudodistance can be constructed. Then for any $x,y\in\Omega\cap U$, there exists $z\in[x,y]$ such that
 \begin{align*}
 \delta_{\Omega}(z)\gtrsim M(x,y)
 \end{align*}
\end{lemma}
\begin{rmk}
This Lemma is an improvement of Lemma 3.1 in \cite{2022Bi}.
\end{rmk}
\begin{proof}
First, fix a point $\omega\in\Omega\cap U$. Choose $z\in[x,y]$ such that $K_{\Omega}(x,z)=(y|\omega)_{x}$ and $K_{\Omega}(y,z)=(x|\omega)_{y}$, we show that
$$(x|y)_{\omega}\geq K_{\Omega}(\omega,z)-2\delta.$$
Here we use the same $\delta$ as in the $\delta$-thin property.
By the proof of the above lemma we can find $z'\in[x,\omega]$  such that $K_{\Omega}(z,z')\leq \delta$.
Since
$$K_{\Omega}(z,\omega)\leq K_{\Omega}(z,z')+K_{\Omega}(z',\omega)$$
and
$$K_{\Omega}(z,x)\leq K_{\Omega}(z,z')+K_{\Omega}(z',x)$$
thus we have
\begin{align*}
K_{\Omega}(z,\omega)&\leq 2K_{\Omega}(z,z')+K_{\Omega}(z',\omega)+K_{\Omega}(z',x)-K_{\Omega}(z,x)\\
&\leq K_{\Omega}(x,\omega)-K_{\Omega}(z,x)+2\delta.
\end{align*}
Similarly, we can find $z''\in[y,\omega]$  such that $K_{\Omega}(z,z'')\leq \delta$ and we have
$$K_{\Omega}(z,\omega)\leq K_{\Omega}(y,\omega)-K_{\Omega}(z,y)+2\delta$$
which gives
\begin{align}\label{est7}
(x|y)_{\omega}\geq K_{\Omega}(\omega,z)-2\delta.
\end{align}
On the other hand, by Theorem \ref{main}, we have
\begin{align}\label{visual}
(x|y)_{\omega}&= \frac{1}{2}\log \frac{(M(x,\omega)+|r(x)|\vee |r(\omega)|)(M(y,\omega)+|r(y)|\vee |r(\omega)|)}{(M(x,y)+|r(x)|\vee |r(y)|)|r(\omega)|}\pm C\\
\nonumber &\leq\frac{1}{2}\log\frac{1}{M(x,y)}\pm C.
\end{align}
Also, by Lemma \ref{est4},
\begin{align}\label{est8}
K_{\Omega}(\omega,z)\geq\frac{1}{2}\log\frac{\delta_{\Omega}(\omega)}{\delta_{\Omega}(z)}.
\end{align}
Thus, by (\ref{est7}), (\ref{visual}) and (\ref{est8}) we have $$\delta_{\Omega}(z)\gtrsim M(x,y).$$
\end{proof}
Actually, by (\ref{pseu}) we have
\begin{align}\label{visual2}
\delta_{\Omega}(z)\gtrsim|x-y|^{L}.\end{align}
 \begin{thm}\label{gromov4}
 Let $\Omega$ be a bounded convex domain with smooth boundary of finite type, then $(\Omega,K_{\Omega})$ is a Gromov hyperbolic space.
 \end{thm}
 \begin{proof}
% Firstly, for any point $p\in\partial\Omega$, by Theorem \ref{main} there exists a open neighbourhood $U$ such that the estimates hold for every points in $\Omega\cap U.$ Similar to the proof in, we will show that $\forall x,y,z,\omega\in\Omega\cap U$,
% $$(x|y)_{\omega}\geq(x,z)_{\omega}\wedge(z,y)_{\omega}-C.$$
% or equivalently,
% \begin{align}
% K_{\Omega}(x,y)+K_{\Omega}(z,\omega)\leq (K_{\Omega}(x,z)+K_{\Omega}(y,\omega))\vee (K_{\Omega}(x,\omega)+K_{\Omega}(y,z))+C
% \end{align}
% Suppose we are given numbers $r_{ij}\geq 0$ such that $r_{ij}\leq Cr_{ji}$ and $r_{ij}\leq C(r_{ik}+r_{kj})$ for $i,j,k\in\{1,2,3,4\}.$ Then we claim that
% $r_{12}r_{34}\leq 4C^4 (r_{13}r_{24}\vee r_{14}r_{23})$.
% We may assume that $r_{13}$ is the smallest of the quantities $r_{ij}$ on the right hand side of the inequality. Then
% $$r_{12}\leq C(r_{13}+r_{32})\leq C(C+1)r_{23}\leq 2C^2 r_{23}$$ and $r_{34}\leq C(r_{31}+r_{14})\leq 2C^2 r_{14}$, then the inequality follows.
%
% Now let $x_i,i\in\{1,2,3,4\}$, be four arbitrary points in $\Omega\cap U$, and $h_i=|r(x_i)|$. Set $d_{ij}=M(x_i,x_j)$ and $r_{ij}=d_{ij}+h_i\vee h_j.$ Then
% \begin{align*}
% (d_{12}+&h_1\vee h_2)(d_{34}+h_3\vee h_4)\leq\\
% & 4C^4 (d_{13}+h_1\vee h_3)(d_{24}+h_2\vee h_4)\vee (d_{14}+h_1\vee h_4)(d_{24}+h_2\vee h_3)
%  \end{align*}
%  By Theorem \ref{main} it translates to (\ref{gromov}).
%We can also choose open set $V$ such that $p\in V\subset\subset U$.

  By Lemma \ref{gromov2} and the compactness of $\partial\Omega$, we claim we can choose finite open sets $\{U_{i}\}_{i=1}^{N}$ and $\{V_{i}\}_{i=1}^{N}$ such that $V_{i}\subset\subset U_{i}$ and $\partial\Omega\subset\bigcup\limits_{i=1}^{N}V_{i}$ and for every $U_{i}$, there exists $C_{i}>0$ such that $\forall x,y,z,\omega\in\Omega\cap U_{i}$,
 $$(x|y)_{\omega}\geq(x,z)_{\omega}\wedge(z,y)_{\omega}-C_{i}.$$

Let $K=\Omega/ \bigcup\limits_{i=1}^{N} V_{i}$, then $K$ is a compact set in $\Omega$.
We can choose an open set $V_{N+1}\subset\subset\Omega$, $V_{N+1}\supset K$, such that $\forall i\in\{1,\cdots,N\}$ if $x\in V_{i}$ and $y\notin U_{i}$, then $[x,y]\cap U_{i}\cap V_{N+1}\neq \emptyset.$
In fact, by (\ref{visual2}), there is $\alpha_i>0$ such that for any $x,y\in U_{i}$ for some $i\in\{1,\cdots,N\}$, $\exists z\in[x,y]$ satisfying $\delta_{\Omega}(z)\geq \alpha_{i}|x-y|^{L}$. Let
$$A=\min\{|x_1-x_2|:x_1\in\partial \overline{U_{j}}\cap \overline{\Omega},x_2\in \overline{V_j}\cap \overline{\Omega},j=1,\cdots, N\}$$
and $\alpha=\min\{\alpha_{i}:i=1,\cdots,N\}$.
Then, we just choose $V_{N+1}$ which satisfies $V_{N+1}\supset K$ and $V_{N+1}\supset \{z\in\Omega:\delta_{\Omega}(z)\geq \alpha A^{L}\}$.
For convenience we denote $U_{N+1}=V_{N+1}.$

Denote $M=\max\{K_{\Omega}(x,y):x,y\in \overline{V}_{N+1}\}$ and $C=\max\{C_{i}:i=1,\cdots,N\}$.  We will show that for any $x,y,z\in\Omega$,  the geodesic triangle $\triangle xyz$ has the  $(12C+2M)$-thin triangle property.

 (1). Suppose $x,y,z\in V_{i}$ for some $i\in\{1,\cdots,N+1\}$.

  When $i\in\{1,\cdots,N\}$, by Lemma \ref{gromov3} we know that each side of the geodesic triangle is contained in a $4C$-neighbourhood of the other two sides.

  When $i=N+1$, obviously each side of the geodesic triangle is contained in $2M$-neighbourhood of the other two sides.
  
Thus the geodesic triangle $\triangle xyz$ has the $(2M+4C)$-thin triangle property.

 % Since
%  $$K_{\Omega}(x,y)=(y|z)_{x}+(x|z)_{y}$$
%  $\forall u\in[x,y]$, we have $K_{\Omega}(x,u)\leq(y|z)_{x}$ or $K_{\Omega}(y,u)\leq (x|z)_{y}.$
%  Suppose $K_{\Omega}(x,u)\leq(y|z)_{x}$, since $(y|z)_{x}\leq K_{\Omega}(x,z)$ by the triangle inequality, we can choose a point $v\in[x,z]$ such that $K_{\Omega}(x,v)=K_{\Omega}(x,u)$. Then
%$$(u|v)_{x}=\frac{1}{2}(K_{\Omega}(x,u)+K_{\Omega}(x,v)-K_{\Omega}(u,v))=K_{\Omega}(x,u)-\frac{1}{2}K_{\Omega}(u,v)$$
%  on the other hand
%  \begin{align*}
%  (u|v)_{x}&\geq (u|y)_{x}\wedge(y|z)_{x}\wedge(z|v)_{x}-2C\\
%  &\geq K_{\Omega}(x,u)-2C
%  \end{align*}
%  Thus,
%  $$K_{\Omega}(u,v)\leq 4C.$$
%  Similarly suppose $K_{\Omega}(y,u)\leq(x|z)_{y}$, we can find $v\in[y,z]$ and $K_{\Omega}(u,v)\leq 4C.$

  (2). Suppose $x,y\in V_{i}$ for some $i\in\{1,\cdots,N+1\}$ and $z\in V_{j}$ for some $j\in\{1,\cdots,N+1\},j\neq i$. We may also assume that $z\notin U_{i}$, else $x,y,z\in U_{i}$ and one can repeat the proof in case (1).

  If $i,j\neq N+1$, then $$[x,z]\cap \partial V_{N+1}\cap U_{i}\neq\emptyset$$
   and
   $$[x,z]\cap \partial V_{N+1}\cap U_{j}\neq\emptyset$$
   by the choice of $V_{N+1}.$
   Write the geodesic $[x,z]$ to be $\gamma_1(t),t\in[0,1], \gamma_1(0)=x$ and $ \gamma_1(1)=z$. Note the set
  $$T_1=\{t\in[0,1]:\gamma_1(t)\in\partial V_{N+1}\cap U_i\}$$
  is not empty.
  Let $t_1=\min\{t\in T_1\}$ and $x'=\gamma_1(t_1)$. Let $$T_2=\{t\in[0,1]:\gamma_1(t)\in\partial V_{N+1}\cap U_j\}\neq\emptyset,$$ let $t_2=\max\{t\in T_2\}$ and $x''=\gamma_1(t_2)$. By considering the geodesic $[y,z]$ we can define $y'$ and $y''$ similarly, then $x',y'\in U_{i}\cap\partial V_{N+1}$ and $x'',y''\in U_{j}\cap\partial V_{N+1}$.

  Then by case (1) we can see that each side of the geodesic quadrilateral $xx'y'y$ lies in the $8C$-neighbourhood of the other sides and each side of the geodesic triangle $\triangle zx''y''$ lies in the $4C$-neighbourhood of the other sides. Besides, all of $x',y',x'',y''$ lie in $\overline{V}_{N+1}$, which means each side of the geodesic quadrilateral $x'y'x''y''$ lies in the $2M$-neighbourhood of the others.
   Thus,  the geodesic triangle $\triangle xyz$ has the $(12C+2M)$-thin triangle property.

   If $i=N+1$ or $j=N+1$, the proof is similar.

  (3). Suppose $x,y,z$ lies in different $V_{i}, V_{j}, V_{k},i,j,k\in\{1,\cdots,N+1\}.$

  The proof is similar to the proof in case (2). Thus we complete the proof.
 \end{proof}
 \begin{rmk}
 Note that for fixed $\omega\in\Omega \cap V_{i}$, if $x,y\in\Omega \cap V_{i}$, we have
 $$(x|y)_{\omega}= \frac{1}{2}\log \frac{(M_i(x,\omega)+|r(x)|\vee |r(\omega)|)(M_i(y,\omega)+|r(y)|\vee |r(\omega)|)}{(M_i(x,y)+|r(x)|\vee |r(y)|)|r(\omega)|}\pm C$$
 Now suppose a sequence of $\{x_k\}$ in $(\Omega,K_{\Omega})$ is a Gromov sequence, then by taking a subsequence if necessary, we know there exists $i\in\{1,\cdots,N\}$, such that $\{x_{k_{j}}\}\in V_{i}.$ Note that for fixed $\omega\in V_i \cap \Omega$, by Lemma \ref{visibility} and the triangle inequality, we have there exists $C>0$ such that for any $x\in V_i$ and $y\in \Omega/U_i$ (and so $[x,y]\cap \partial U_i\neq \emptyset$), we have
 $$(x|y)_{\omega}<C$$
 which means $\{x_{k}\}$ is contained in $U_{i}$ when $k$ large enough. Then $M_i(x_{m},x_{l})\rightarrow 0$ and $|r(x_{m})|\rightarrow 0$. This happens if and only if $x_{k}\rightarrow x_0\in\partial\Omega \cap \overline{V}_{i}$. Moreover, two sequences are equivalent if and only if their limits points on $\partial\Omega$ are the same. Thus we can identify the Gromov boundary $\partial_{G}\Omega$ with the Euclidean boundary as sets.

 Furthermore, for fixed $\omega\in V_{i}$ then for any $\xi,\eta \in\partial\Omega\cap V_{i}$,
 \begin{align}\label{visual3}
 e^{-(\xi|\eta)_{\omega}}\asymp M_{i}(\xi,\eta)
 \end{align}
 which means on $\partial \Omega\cap V_i$, $M_{i}$ is  snowflake equivalent to visual metrics on $\partial\Omega$.
 \end{rmk}

 Now we introduce the definition of rough quasi-isometric maps as follows.
\begin{defn}
Let $f: X\to Y$ be a map between metric spaces $X$ and $Y$, and let $\lambda\geq 1$ and $k\geq 0$ be constants. \begin{enumerate}
\item
If for all $x,\: y\in X$,
$$\lambda^{-1}d_X(x,y)-k\leq d_Y(f(x),f(y))\leq \lambda d_X(x,y)+k,$$
then $f$ is called an {\it $(\lambda, k)$-rough quasi-isometry} (cf. \cite{Schramm1999Embeddings}).
%\noindent
\item If $$\lambda d_X(x,y)-k\leq d_Y(f(x),f(y))\leq \lambda d_X(x,y)+k,$$
then $f$ is called an {\it $(\lambda, k)$-rough similarity.}
\item If $$ d_X(x,y)-k\leq d_Y(f(x),f(y))\leq  d_X(x,y)+k,$$
 then $f$ is called an {\it $k$-rough isometry}.
\end{enumerate}

\end{defn}
 \begin{defn}
Let $f: X\to Y$ be a bijection between metric spaces $X$ and $Y$, and $\lambda\geq 1,\alpha>0$ be a constant.
\item If for all $x,y\in X$
$$\frac{1}{\lambda} d_X(x,y)\leq d_Y(f(x),f(y))\leq \lambda d_X(x,y),$$
then $f$ is $\lambda$-bilipschitz; if
\item $$\frac{1}{\lambda} d_X(x,y)^{\alpha}\leq d_Y(f(x),f(y))\leq \lambda d_X(x,y)^{\alpha},$$
then $f$ is an $(\lambda,\alpha)$-snowflake map.
\item If for distinct points $x,y,z\in X$,
$$\frac{d_{Y}(f(x),f(y))}{d_{Y}(f(x),f(z))}\leq \eta_{\alpha,\lambda}\left(\frac{d_{X}(x,y)}{d_{X}(x,z)}\right).$$
then $f$ is an $(\lambda,\alpha)$-power quasisymmetry. Here
\begin{align*}
\eta_{\alpha,\lambda}(t)=\left\{
\begin{array}{ll}
\lambda t^{1/\alpha} & 0<t<1\\
\lambda t^{\alpha} & t\geq 1
\end{array}
\right.
\end{align*}
\end{defn}
\begin{proposition}[Section 6, \cite{Schramm1999Embeddings}]\label{ext}
Suppose $f: X\to Y$ be a rough quasi-isometry between Gromov hyperbolic metric spaces $X$ and $Y$. Then $f$ induces a power quasisymmetry $\tilde{f}:\partial_{G}X\rightarrow \partial_{G} Y$ with respect to the visual metrics. If $f$ is a rough similarity, then $\tilde{f}$ is a snowflake map. If $f$ is a rough isometry and $\partial_{G}X$ and $\partial_{G}Y$ are equipped with visual metrics satisfying (\ref{ee}) with the same number $\epsilon>0$, then $\tilde{f}$ is bilipschitz.
\end{proposition}
The following corollary follows from Theorem \ref{gromov4}, Proposition \ref{ext} and (\ref{visual3}).
 \begin{cor}
 Let $\Omega_{i}(i=1,2)$ be bounded convex domains with smooth boundary of finite type $L$, and let $f:\Omega_1\rightarrow\Omega_2$ be a rough quasi-isometry map. Then $f$ extends continuously to a map $\bar{f}:\bar{\Omega}_1\rightarrow\bar{\Omega}_2$. Suppose $p_1\in\partial\Omega_1$, $p_2=\bar{f}(p)\in\partial\Omega_2$. Then there exist open neighbourhoods $U_i(i=1,2)$ of $p_i$, the local pseudodistance $M_i$ defined in $U_i$, some $\lambda>1 $ and $\alpha>0$ such that for any $\xi,\eta,\theta\in\partial \Omega\cap U_1$ 
 $$\frac{M_2(\bar{f}(\xi),\bar{f}(\eta))}{M_2(\bar{f}(\xi),\bar{f}(\theta))}\leq  \eta_{\alpha,\lambda}\left(\frac{M_1(\xi,\eta)}{M_1(\xi,\theta)}\right).$$
 If $f$ is a rough similarity, then
 $$\frac{1}{\lambda}M_1(\xi,\eta)^{\alpha}\lesssim M_2(\bar{f}(\xi),\bar{f}(\eta))\lesssim \lambda M_1(\xi,\eta)^{\alpha}.$$
 If $f$ is a rough isometry, then
 $$M_1(\xi,\eta)\lesssim M_2(\bar{f}(\xi),\bar{f}(\eta))\lesssim M_1(\xi,\eta)$$

 \end{cor}
\textbf{	Proof of Corollary \ref{1.3}:}
 	
 	\bigskip
 	
Since $f$ is holomorphic, for any $x,y\in\Omega_1$,
 $$K_{\Omega_2}(f(x),f(y))\leq K_{\Omega_1}(x,y).$$

 We first prove that $f$ extends to the boundary continuously. Let $U_i(i=1,2)$ be the open neighbourhoods of $p_i$, $M_i$ be the local pseudodistance defined in $U_i$. By (\ref{visual2}) , take $ V_2\subset\subset U_2$ then for any $x\in V_2$ and $y\notin U_2$, there exists $z\in[x,y]$ such that $\delta_{\Omega_2}(z)\gtrsim C$ where $$C=\inf\{|a-b|^{L}:a\in \partial U_2\cap \Omega,b\in V_2\}.$$ Then suppose there are two sequences $\{x_n\}$ and $\{y_n\}$ converging to $\xi\in\partial\Omega_1\cap U_1$. Assume that $f(x_n)\rightarrow \xi'\in\partial\Omega_2\cap U_2$ and
 $f(y_n)\rightarrow \eta'(\neq \xi')\in\partial\Omega_2$.

 We claim that $\eta'\in\partial\Omega_2\cap U_2$.
 If $\eta'\notin\partial\Omega_2\cap U_2$, consider the geodesic in $\Omega_2$ joining $f(x_n)$ and $f(y_n)$. We have $[f(x_n),f(y_n)]\cap \partial U_2\neq \emptyset$ for sufficient large $n$. Then take $\omega_n\in[f(x_n),f(y_n)]$ satisfying $\delta_{\Omega_2}(\omega_n)\gtrsim C$. By Lemma \ref{est4} and the Triangle inequality we have
 $$K_{\Omega_2}(f(x_n),f(y_n))\geq \frac{1}{2}\log\frac{C^2}{\delta_{\Omega_2}(f(x_n))\delta_{\Omega_2}(f(y_n))}.$$
 On the other hand, by Theorem \ref{main} and (d) in Lemma \ref{dist} we have 
 \begin{align*}
  K_{\Omega_2}(f(x_n),f(y_n))\leq K_{\Omega_1}(x_n,y_n)\leq \frac{1}{2}\log\frac{(M_1(x_n,y_n)+\delta_{\Omega_1}(x_n)\vee\delta_{\Omega_1}(y_n))^{2}}{\delta_{\Omega_1}(x_n)\delta_{\Omega_1}(y_n)}+C.
 \end{align*}
  Since $f$ is proper and $\Omega_i(i=1,2)$ are convex,
 by Hopf lemma(see Proposition 12 in \cite{article} for details), 
 $$\delta_{\Omega_2}(f(x))\leq\delta_{\Omega_1}(x)$$
 which causes a contradiction between the upper and lower bounds for $K_{\Omega_2}(f(x_n),f(y_n))$. So $\eta'\in\partial\Omega_2\cap U_2$ and we have
 $$M_2(f(x),f(y))\lesssim M_1(x,y)+\delta_{\Omega_1}(x)\vee\delta_{\Omega_1}(y).$$
Then it follows that $f$ extends continuously to the boundary. Let $V_1=U_1\cap f^{-1}(V_2)$ then for any $\xi,\eta\in\partial\Omega_1\cap \overline{V}_1$, applying Theorem \ref{main} and we have
 $$M_2(\bar{f}(\xi),\bar{f}(\eta))\lesssim M_{1}(\xi,\eta).$$\qed
\bigskip
\bibliography{reference}
\bibliographystyle{plain}{}
\end{document}